\documentclass{endm}
\usepackage{endmmacro}

\usepackage{epsfig}
\usepackage{pstricks}
\usepackage{url}
\usepackage{hyperref}
\usepackage{amsmath}

\newcommand{\graph}{\operatorname{G}}

\newcommand{\lcm}{\operatorname{lcm}}

\newcommand{\simp}{{\textrm{s}}}

\newcommand{\F}{\mathbb F}

\newcommand{\N}{\mathbb N}

\begin{document}

\begin{frontmatter}

\title
%[Maximum diameter of complexes and pseudo-manifolds]
{The maximum diameter of pure simplicial complexes and pseudo-manifolds}

\author{Francisco Criado and
Francisco Santos%
\thanksref{fsgrant}}

\address%[F.~Santos]
{
{Departamento de Matem\'aticas, Estad\'istica y Computaci\'on}\\
{Universidad de Cantabria, E-39005 Santander, Spain}
}

\thanks[fsgrant]
{Work of F.~Santos is supported in part by the Spanish Ministry of Science (MICINN) through grant MTM2014-54207P.
E\_mail: fcriado92@gmail.com, francisco.santos@unican.es}

%\date{\today}

\begin{abstract}
We construct $d$-dimensional pure simplicial complexes and pseudo-manifolds (without boundary) with $n$ vertices whose combinatorial diameter grows as $c_d n^{d-1}$ for a constant $c_d$ depending only on $d$, which is the maximum possible growth. Moreover, the constant $c_d$ is optimal modulo a singly exponential factor in $d$. The pure simplicial complexes improve on a construction of the second author that achieved $c_d n^{2d/3}$. For pseudo-manifolds without boundary, as far as we know, no construction with diameter greater than $n^2$ was previously known.
\end{abstract}

%\keywords{
%}
%\subjclass[2000]{52B05, 52B12, 90C60, 90C05}

%\thispagestyle{empty}\tableofcontents\clearpage\setcounter{page}{1}

%\maketitle

\end{frontmatter}

%\tableofcontents

\section{Introduction}

A \emph{pure simplicial complex of dimension $d-1$} (or a \emph{$(d-1)$-complex}, for short) is any family $C$ of $d$-element subsets of a set $V$ (typically, $V=[n]:=\{1,\dots,n\}$). Elements of $C$ are called \emph{facets} and any subset of a facet is a \emph{face}\footnote{The standard usage is to consider all faces, not only facets, as elements of $C$, and then call facets the maximal ones. Our approach is equivalent and, for our purposes, simpler.}. More precisely, a $k$-face is a face with $k+1$ elements. Faces of dimensions $0$, $1$, and $d-2$ are called, respectively, \emph{vertices}, \emph{edges} and \emph{ridges} of $C$.  Observe that a pure $(d-1)$-complex is the same as a \emph{uniform hypergraph of rank $d$}. Its facets are called \emph{hyperedges} in the hypergraph literature.

The \emph{adjacency graph} or \emph{dual graph} of a pure simplicial complex $C$, denoted $\graph(C)$, is the graph having as vertices the facets of $C$ and as edges the pairs of facets $X,Y\in C$ that differ in a single element (that is, those that share a ridge). 
Complexes with a connected adjacency graph are called \emph{strongly connected}. The \emph{combinatorial diameter} of $C$ is the diameter, in the graph theoretic sense, of $\graph(C)$.

We are interested in how large can the diameter of a pure simplicial complex be in terms of its dimension and number of vertices. For this we set:
\[
\begin{tabular}{rl}
$H_\simp(n,d):=$& maximum diameter of pure strongly connected \\
& $(d-1)$-complexes with $n$ vertices.
\end{tabular}
\]

The function $H_\simp(n,d)$ is known to be exponential in $d$:
\begin{theorem}[Santos~\protect{\cite[Corollary 2.12]{Santos:progress}}]
\label{thm:Santosbound}
\[
\Omega\left(\frac{n}{d}\right)^{\frac{2d}{3}} \le H_\simp(n,d) \le \frac{1}{d-1}\binom{n}{d-1}\simeq \frac{n^{d-1}}{d!}.
\]
\end{theorem}

The upper bound is obtained by  counting the possible number of ridges, while the lower bound comes from a construction using the \emph{join} operation. Another construction giving a lower bound of type $n^{\frac{d}{4}}$ is contained in~\cite[Thm.~4.4]{Kim:abstractions}. In this short note we show a simple and (relatively) explicit construction giving

\begin{theorem}
\label{thm:pure_complexes}
For every $d\in \N$ there are infinitely many $n\in\N$ such that:
\[
  H_\simp(n,d) \ge \frac{n^{d-1}}{(d+2)^{d-1}}-3.
\]
\end{theorem}
Observe that this matches the upper bound in Theorem~\ref{thm:Santosbound}, modulo a factor %of $\frac{\sqrt{2\pi(d+2)}(d+2)^2}{(d+1)\e^{d+2}}$
in $\Theta(d^{3/2}e^{-d})$, since $d!\simeq e^{-d}d^d \sqrt{2\pi d}$.

\begin{remark}
Our proof of Theorem~\ref{thm:pure_complexes} uses an arithmetic construction valid only when the number $n$ of vertices is of the form $q(d+2)$ for a sufficiently large prime power $q$. 
But every interval $[m,2m]$ contains an $n$ of that form, because there is a power of $2$ between $m/(d+2)$ and $m/2(d+2)$). 
Hence, the theorem is also valid ``for every $d$ and sufficiently large $n$'', modulo an extra factor of $2^{d-1}$ in the denominator.
\end{remark}

Motivation for this question and relatives of it comes from the Hirsch Conjecture which, written in the language of simplicial complexes, said: \emph{The maximum diameter of a polytopal simplicial $(d-1)$-sphere with $n$ vertices cannot exceed $n-d$.} Here, a \emph{simplicial sphere} is a pure simplicial complex whose underlying topological space is homeomorphic to a sphere. A \emph{polytopal sphere} is the simplicial complex of proper faces of a simplicial polytope. Although the original Hirsch conjecture has been disproved~\cite{Santos:Hirsch-counter}, the only counter-examples to it that we know of exceed the conjectured diameter by a small fraction. In particular, the following polynomial version of the Hirsch conjecture is open, even in the linear case (the case $k=1$):
%, where $H_p(n,d)$ denotes the maximum diameter of polytopal spheres of given dimension and number of vertices:

\begin{conjecture}[Polynomial Hirsch Conjecture]
\label{conj:poly_Hirsch}
There are constants $c$, $k$ such that the diameter of every polytopal $(d-1)$-sphere with $n$ vertices is bounded above by $c n^k$.
\end{conjecture}

An approach that has been tried often is to generalize the conjecture to more general complexes than polytopal spheres. Theorem~\ref{thm:Santosbound} shows that generalizing to arbitrary pure complexes is too much, but it is plausible that \emph{simplicial manifolds} still have polynomial diameter. In particular, the polymath3 project~\cite{Kalai:polymath3}  was devoted to (a more abstract and generalized version of) the following conjecture, inspired by the results in~\cite{EHRR:limits-of-abstraction} and which implies Conjecture~\ref{conj:poly_Hirsch} with $k=2$ and $c=1$:

\begin{conjecture}[H\"ahnle, in \cite{Kalai:polymath3}]
\label{conj:polymath3}
The diameter of every \emph{normal} pure $(d-1)$-complex with $n$ vertices is bounded above by $dn$.
\end{conjecture}

Here a pure simplicial complex $C$ is called \emph{normal} if every two facets $F_1$ and $F_2$ are connected in $\graph(C)$ via a path with the property that all facets in the path contain $F_1\cap F_2$. Another important class of pure complexes are \emph{pseudo-manifolds} without boundary%
\footnote{$C$ is a \emph{pseudo-manifold with boundary} if ridges are contained in \emph{at most} two facets, the boundary of $C$ consisting of the ridges lying in only one facet. Standard usage is to say ``pseudo-manifold'' alone meaning ``without boundary'' and ``pseudo-manifold with boundary'' when boundary is allowed. But to avoid confusion we here insist in saying ``without boundary'' when boundary is forbidden.}%
: strongly connected complexes in which every ridge belongs to exactly two facets. For example, Adler and Dantzig~\cite{AdlDan:abstract} call normal pseudo-manifolds without boundary \emph{abstract polytopes}. Here we prove the following:

\begin{theorem}
\label{thm:pseudo_manifolds}
For every strongly connected pure $(d-1)$-dimensional simplicial complex with $n$ vertices and diameter $\delta$ there is a $(d-1)$-dimensional pseudo-manifold without boundary with $2n$ vertices and diameter at least $\delta+2$.
\end{theorem}

Together with Theorem~\ref{thm:pure_complexes} this implies the following, where $H_{pm}(n,d)$ denotes the maximum diameter of $(d-1)$-dimensional pseudo-manifolds without boundary. As far as we know this is the first construction of pseudo-manifolds without boundary and of exponential diameter. 

\begin{corollary}
\label{cor:pseudo-manifolds}
For every $d\in \N$ there are infinitely many $n\in\N$ such that:
\[
  H_{pm}(n,d) \ge \frac{n^{d-1}}{(2(d+2))^{d-1}}-1.
\]
\end{corollary}

In a similar spirit, Todd~\cite{Todd:duoids} defined \emph{semi-duoids} as the pure simplicial complexes in which every ridge lies in an even number of facets, and called \emph{duoids} the semi-duoids that do not properly contain other semi-duoids. Semi-duoids were later called \emph{oiks} (as a short-hand for ``Euler complexes'') by Edmonds~\cite{Edmonds:oiks} (see also~\cite{VeghStengel}). 
One of the results in~\cite{Todd:duoids} is the construction of duoids with quadratic diameter. Since every pseudo-manifold without boundary is a duoid, Corollary~\ref{cor:pseudo-manifolds} significantly improves that construction.

\section{Proof of Theorem~\ref{thm:pure_complexes}}

Our construction is arithmetic, and uses the following well-known result that can be found, for example, in~\cite[Theorem 33.16]{LidlPilz}:

%\marginpar{Reference ``Applied Abstract Algebra''}
%% Are you sure? Perhaps better "Algebraic Coding Theory", Berlekamp
\begin{theorem}
\label{thm:sequence}
Let $p(x) = x^{d} + a_1 x^{d-1} +\dots+ a_d$ be a primitive polynomial of degree $d$ over the field $\F_q$ with $q$ elements, for some $d\in \N$ and some prime power $q$. Consider the sequence $(u_n)_{n\in \N}$ defined by the linear recurrence
\[
u_{n+d} + a_1 u_{n+d-1} +\dots+ a_d u_n =0,
\]
starting with any non-zero vector $(u_1, \dots,u_{d}) \in \F_q^d$. Then, $(u_n)_{n\in \N}$ has period $q^d-1$. In particular, its intervals of length $d$ cover all of $ \F_q^d\setminus \{(0,\dots,0)\}$. That is:
\[
\left\{(u_i,\dots,u_{i+d-1}) : i\in \{1,\dots, q^d-1\} \right\}= \F_q^d \setminus \{(0,\dots,0)\}.
\]
\end{theorem}

Remember that a primitive polynomial of degree $d$ is the minimal polynomial of a primitive element in the degree $d$ extension $\F_{q^d}$ of $\F_d$. The number of monic primitive polynomials of degree $d$ over $\F_q$ equals $\phi(q^d-1)/d$, since $\F_{q^d}$ has $\phi(q^d-1)$ primitive elements, and each primitive polynomial is the minimal polynomial of $d$ of them.
In our construction we will need the coefficients of $p(x)$ to be all different from zero. Primitive polynomials with this property do not exist for all $q$, but they exist when $q$ is sufficiently large with respect to $d$, which is enough for our purposes:
%We will use the following result:

\begin{lemma}
  For every fixed $d\in\N$ and every sufficiently large prime power $q$, there is a primitive polynomial of degree $d$ over $\mathbb{F}_q$ with all coefficients different form zero.
\end{lemma}

\begin{proof}
This follows from the fact that the number of primitive monic polynomials of degree $d$ is greater than the number of monic polynomials of degree $d$ with at least one zero coefficient, for $q$ large. 

Indeed, the latter is $q^d-(q-1)^d\leq dq^{d-1}$. The former equals $\phi(q^d-1)/d$, which is greater than $(q^d-1)^{1-\epsilon}/d$, for every $0<\epsilon<1$ and sufficiently large $q$. Letting $\epsilon=\frac{1}{d^2}$ we get:
  \[
    \frac{\phi(q^d-1)}{d}>\frac{(q^d-1)^{1-\frac{1}{d^2}}}{d}>\frac{(q^d/2)^{1-\frac{1}{d^2}}}{d}=\frac{q^{(d^2-1)/d}}{2^{1-\frac{1}{d^2}}d}>\frac{q^{d-\frac{1}{d}}}{2d} >dq^{d-1}.
  \]
\end{proof}

With this we can now show our first construction proving Theorem~\ref{thm:pure_complexes}.

\begin{theorem}
\label{thm:lfsr}
Suppose that $p(x)\in \F_q[x]$ is a primitive polynomial of degree $d-1$ with no zero coefficients. Then, there is a pure simplicial complex $C$ of dimension $d-1$, with $n=(d+2)q$ vertices and at least $\frac{n^{d-1}}{(d+2)^{d-1}}-1$ facets whose dual graph is a cycle.
\end{theorem}

\begin{proof}
Our set of vertices is $V=\F_q \times [d+2]$. That is, we have as vertices the elements of $\F_q$ but each comes in $d+2$ different ``colors''. In the sequence $(u_i)_{i\in \N}$ of Theorem~\ref{thm:sequence} we color its terms cyclically. That is, call
\[
v_i=(u_i,\ n \mod(d+2)).
\]
Let $C$ be the simplicial complex consisting of the intervals of length $d$ in the sequence $(v_i)_{i\in \N}$. That is, we let:
\[
F_i= \{v_i,\dots,v_{i+d-1}\}, \qquad
C= \left\{F_i : i\in \{1,\dots, q^d-1\} \right\}.
\]
Observe that the sequence $\{F_i\}_{i\in \N}$ is periodic of period $\lcm\{q^d-1, d+2\} \ge q^d-1=\frac{n^{d-1}}{(d+2)^{d-1}}-1$. Also, by construction, $\graph(C)$ contains a Hamiltonian cycle. We claim that, in fact, $\graph(C)$ equals that cycle. 

For this, observe that ridges in  $C$ are of two types: some are of the form $\{v_i,\dots,v_{i+d-2}\}$ and some are of the form $\{v_i,\dots,v_j,v_{j+2},\dots,v_{i+d-1}\}$. We will study the facets that these types of ridges may belong to.

For a ridge $R=\{v_i,\dots,v_{i+d-2}\}$ to be contained in a facet $F$ we need the color of the vertex in $F\setminus R$ to be either $i-1$ or $i+d-1$ (modulo $d+2$).

Once we have  the color $c$ of the new vertex $v=(u,c)\in F\setminus R$, the recurrence relation (and the fact that $p$ has non-zero coefficients) gives us only one choice for $u$. Thus, $R$ is only contained in the two contiguous facets $F_{i-1}$ and $F_i$.

The same argument applies to a ridge $\{v_i,\dots,v_j,v_{j+2},\dots,v_{i+d-1}\}$. Now the color of the new vertex must be $j+1\mod d+2$ and the recurrence relation implies the vertex to be precisely $v_{j+1}$.
\end{proof}

\medskip
\noindent
\textbf{Proof of Theorem~\ref{thm:pure_complexes}.}\ 
Delete a  facet in the complex $C$ of Theorem~\ref{thm:lfsr}.
\qed
\medskip

A complex whose dual graph is a path, such as the one in this proof, is called
a \emph{corridor} in \cite{Santos:progress}. It is a general fact that the maximum diameter $H_\simp(n,d)$ is always attained at a corridor (\cite[Corollary 2.7]{Santos:progress}). That is to say, $H_\simp(n,d)$ equals the maximum length of an induced path in the \emph{Johnson graph} $J_{n,d}$: the dual graph of the complete complex of dimension $d-1$ with $n$ vertices. Induced paths in graphs are sometimes called \emph{snakes}.  In this language Theorem~\ref{thm:pure_complexes} can be restated as:

\begin{theorem}
There is a constant $c>0$ such that for every fixed $d$ and sufficiently large $n$ the Johnson graph $J_{n,d}$ contains snakes passing through a fraction $c^{-d}$ of its vertices.
\end{theorem}

A stronger statement is known for the graph of a $d$-dimensional hypercube: it contains snakes passing through a positive, independent of $d$, fraction of the vertices~\cite{AbbottKatchalski}.

%Here we call \emph{closed corridors} the complexes whose dual graph is a cycle, such as the one in Theorem~\ref{thm:lfsr}.

\section{Proof of Theorem~\ref{thm:pseudo_manifolds}}

Let $C$ be the simplicial complex in the statement and $V$ its vertex set. By \cite[Corollary 2.7]{Santos:progress} there is no loss of generality in assuming that $C$ is a \emph{corridor}. That is, its dual graph is a path, so its facets come with a natural order $F_0.\dots, F_\delta$. 
%(Incidentally, observe that the complexes constructed in Theorem~\ref{thm:lfsr} are corridors).

We now construct a simplicial complex $C'$ in the vertex set $V'=V\times\{1,2\}$. For a vertex $v\in V$ we denote $v^1$ and $v^2$ the two copies of it in $V'$, and refer to the superscripts as ``colors''.
%Observe that for every $F_i \in C$ there are two special ridges connecting it to $F_{i-1}$ and $F_{i+1}$. 
Let $a_i$ and $b_i$ be the unique vertices in $F_i\setminus F_{i+1}$ and $F_i\setminus F_{i-1}$, respectively. (For $F_0$ and $F_{\delta}$ we choose $a_0$ and $b_\delta$ arbitrarily, but different from $b_0$ and $a_\delta$). We define $C'$ as the complex containing, for each $F_i$, the $2^{d-1}$ colored versions of it in which $a_i$ and $b_i$ have the same color. The diameter of $C'$ is at least the same as that of $C$. Let us see that $C'$ is almost a pseudo-manifold:
\begin{itemize}
\item If a ridge $R$ in $C'$ is obtained from a colored version of $F_i$ by removing a vertex $v$ different from $a_i$ or $b_i$, then the only other facet containing $R$ is the copy of $F_i$ in which the color of $v$ is changed to the opposite one. This is so because the ``uncolored'' version of $R$ is a ridge of only the facet $F_i$ of $C$, by assumption.

\item If a ridge $R$ in $C'$ is obtained from a colored version of $F_i$ ($i<\delta$) by removing $a_i$ then the only other facet containing $R$ is obtained by adding to it the vertex $b_{i+1}$ with the same color as $a_{i+1}$ has in $R$.

\item Similarly, if a ridge $R$ in $C'$ is obtained from a colored version of $F_i$ ($i>0$) by removing $b_i$ then the only other facet containing $R$ is obtained by adding to it the vertex $a_{i-1}$ with the same color as $b_{i-1}$ has in $R$.
\end{itemize}

That is, the only ridges of $C'$ that do not satisfy the pseudo-manifold property are the $2^{d-1}$ colored versions of $R_1:=F_0\setminus \{b_0\}$ and the $2^{d-1}$ colored versions of $R_2:=F_\delta\setminus \{a_\delta\}$, which form two $(d-2)$-spheres, (each with the combinatorics of a cross-polytope). Choose a vertex $a$ in $R_1$ and a vertex $b\in R_2$, different from one another (which can be done since $R_1\ne R_2$). Consider the complex $C''$ obtained from $C'$ adding to it all the colored versions of $R_1\setminus a$ joined to $\{a^1,a^2\}$ and all the colored versions of $R_2\setminus b$ joined to $\{b^1,b^2\}$. The effect of this is glueing two $(d-1)$-balls with boundary the two $(d-2)$-spheres we wanted to get rid off, so that $C''$ is now a pseudo-manifold. (Observe that the new ridges introduced in $C''$ all contain either $\{a^1,a^2\}$ or $\{b^1,b^2\}$ so they were not already in $C'$).
\qed

\begin{remark}
\label{rem:closed-corridor}
In some contexts it may be useful to apply Theorem~\ref{thm:pseudo_manifolds} to \emph{closed corridors}, that is, pure complexes whose dual graph is a cycle. The construction in the proof works exactly the same except now $C'$ is already a pseudo-manifold, with no need to glue two additional balls to it as we did in the final step of the proof.
\end{remark}

\end{document}